\documentclass[11pt]{article}
\usepackage{graphicx}
\usepackage{color}
\usepackage{amssymb}
\usepackage{amsmath}
\usepackage{ulem} 

\usepackage[utf8]{inputenc}
\usepackage{amsmath,amsfonts,bm,amsthm}
\usepackage{color}

\usepackage{graphicx}
\usepackage{bm}
\usepackage{float}
\usepackage{hyperref}

\newtheorem{thm}{Theorem}[section] 
\newtheorem{lem}[thm]{Lemma}

\begin{document}
\pagenumbering{arabic}

%

%

{\bf\centerline{\large  Generalized spectrum of second order differential operators: 3D problems}}
\vskip16pt

\centerline{Ivana Pultarov\' a\footnote{Address: Faculty of Civil Engineering,
Czech Technical University in Prague, Th\' akurova~7, Praha~6;
e-mail: ivana.pultarova@cvut.cz}}

\vskip36pt



{\bf Abstract.}
Generalized spectra of differential operators can be related 
to spectra of preconditioned discretized operators.
Obtaining (estimates of) the eigenvalues of the preconditioned
discretized operators may lead to better 
estimating of the quality of preconditioners.
In this short paper, we 
answer the open question posted in the recent paper 
Generalized spectrum of second order differential operators,
authored by
Gergelits, Nielsen, and Strako\v s.
The proof we present allows us to fully extend  
characterizing the 
generalized spectra of $\nabla\cdot K\nabla u=\lambda\triangle u$
to problems
of dimension three or higher.


\section{Itroduction}

Efficiency of numerical solution methods for differential equations
with symmetric operators
usually depends on spectral properties of the underlying matrices.
Different preconditioning techniques 
transform the solved systems to systems with better
spectral characteristics. Recent progress in estimating the spectrum
of preconditioned systems was achieved in~\cite{G}
and then in~\cite{LPZ2020,PL2021}, where new methods for obtaining two sided-bounds 
on all eigenvalues of the preconditioned matrices 
were introduced. These results, however, were motivated by
infinite-dimensional problems.
In~\cite{NTH2009} it was proved that eigenvalues of  
preconditioned operators with scalar coefficients 
defined in infinite-dimensional spaces
correspond to the coefficient throughout the domain.
In~\cite{GNS2020}, problems with tensorial data were considered and 
spectra of preconditioned operators 
were fully characterized. The proof, however,
was presented only for problems of two space variables.
In this paper, we prove the main statement of~\cite{GNS2020} 
for three-dimensional (3D) cases, which can be easily generalized
for problems of higher dimensions. 

In the subsequent section, we first recall the notation and the main result of~\cite{GNS2020}. Then in the second subsection, we 
introduce the proof of characterizing the spectra of preconditioned
operators in 3D. The main result is formulated in Theorem~\ref{t1}, 
the proof of which is based on Lemma~\ref{lem22}, Lemma~\ref{lem23},
Lemma~\ref{lem24}, and Lemma~\ref{lem25}.
A brief discussion concludes the paper.

\section{Characterization of the spectrum}

\subsection{Notation}

We study the generalized eigenvalue problem 
\begin{eqnarray}\label{p1}
\nabla\cdot K(\bm{x})\nabla u(\bm{x})&=&\lambda\triangle u(\bm{x})\quad\text{for}\; \bm{x}\in 
\Omega,\\
u(\bm{x})&=&0\quad\text{for}\; \bm{x}\in \partial\Omega,\nonumber
\end{eqnarray}
where $\Omega\subset{\mathbb R}^d$, $d=1,2,3$, is a bounded domain with
Lipschits boundary $\partial\Omega$ and $K:\Omega\to {\mathbb R}^{d\times d}$ is the coefficient matrix which is assumed to be diagonal,
\begin{equation}\nonumber
K(\bm{x})=\left[\begin{array}{ccc}
\kappa_1(\bm{x})&&\\
&\ddots&\\
&&\kappa_d(\bm{x})\end{array}\right].
\end{equation}
Let us define the operators $\mathcal{L}$, $\mathcal{A}$:
$H_0^1(\Omega)\to H^{-1}(\Omega)$,
\begin{eqnarray}
\langle \mathcal{L}u,v\rangle&=&\int_\Omega \nabla v\cdot\nabla u,
\quad u,v\in H_0^1(\Omega),\nonumber\\
\langle \mathcal{A}u,v\rangle&=&\int_\Omega \nabla v\cdot K\nabla u,
\quad u,v\in H_0^1(\Omega),\nonumber
\end{eqnarray}
and the induced norm $\Vert u\Vert^2_{\mathcal{L}}=\langle \mathcal{L}u,u\rangle$.
The spectrum of the preconditioned operator
$\mathcal{L}^{-1}\mathcal{A}$: $H_0^1(\Omega)\to H_0^1(\Omega)$ 
is defined as
\begin{equation}\nonumber
\text{sp}(\mathcal{L}^{-1}\mathcal{A})=\{\lambda\in {\mathbb C};\;
\lambda \mathcal{I}-\mathcal{L}^{-1}\mathcal{A}\; \text{does not have
a bounded inverse}\}.
\end{equation}
Thus the weak form of problem~\eqref{p1} can be viewed as 
finding the spectrum of 
the preconditioned operator $\mathcal{L}^{-1}\mathcal{A}$.

We assume that $K$ is Lebesgue integrable and essentially bounded
in ${\Omega}$. Instead of the diagonal shape of $K$,
we can also deal with general symmetric matrices $K$. In such cases,
the main results, spectral estimates, remain valid.
To see this, an orthogonal decomposition of $K$ can be applied and 
the approach of~\cite[Section~4]{GNS2020} can be followed.

The main result of~\cite{GNS2020} formulated for 3D case ($d=3$)
reads (see~\cite[Theorem~1.1]{GNS2020})
\begin{thm}\label{t1}
Let $K$ be continuous in $\overline{\Omega}$.
Then the spectrum of the operator $\mathcal{L}^{-1}\mathcal{A}$
equals
\begin{equation}\nonumber
{\rm sp}(\mathcal{L}^{-1}\mathcal{A})={\rm Conv}(\cup_{i=1}^3\kappa_i(\overline{\Omega})).
\end{equation}
\end{thm}
\begin{proof}
The proof of the first inclusion consists of Lemmas~\ref{lem22},
\ref{lem23}, and~\ref{lem24} which are presented in the subsequent section.
 The inverse inclusion follows from 
Lemma~\ref{lem25}.
\end{proof}

\subsection{Proof of the 3D case}

To prove Theorem~\ref{t1} we first introduce or recall four auxiliary lemmas.
\begin{lem}\label{lem22}
Let $i\in\{1,2,3\}$ and let $\kappa_i$ be continuous in $\bm{x}^0\in \Omega$.
Then $\kappa_i(\bm{x}^0)\in {\rm sp}(\mathcal{L}^{-1}\mathcal{A})$.
\end{lem}
\begin{proof}
Without any loss of generality, let $i=1$ and let $\kappa_1$ 
be continuous in $\bm{x}^0\in \Omega$. Denote $\kappa_1(\bm{x}^0)=\lambda$.
We shall construct parametrized functions $v_r\in H^1(\Omega)$
such that
\begin{equation}\label{co}
\lim_{r\to 0}\Vert v_r\Vert_{\mathcal L}\ne 0\quad \text{and}\quad
\lim_{r\to 0}\Vert (\lambda\mathcal{I}-\mathcal{L}^{-1}\mathcal{A})v_r
\Vert_{\mathcal L}= 0.
\end{equation}
For $r\in(0,1)$ define a two-dimensional disc $D_r$, its neighborhood
$R_r$, and a cylinder $C_r$
\begin{eqnarray}\nonumber
D_r&=&\{\bm{x}\in \mathbb{R}^3;\, x_1=x_1^0,\, d([x_2,x_3],[x_2^0,x_3^0])\le r\}\nonumber\\
R_r&=&\{\bm{x}\in \mathbb{R}^3;\, d(\bm{x},D_r)\le r^2\}\nonumber\\
C_r&=&\{\bm{x}\in \mathbb{R}^3;\,x_1\in(x_1^0-r^2,x_1^0+r^2),\, 
d([x_2,x_3],[x_2^0,x_3^0])\le r\},\nonumber
\end{eqnarray}
where $d(A,B)$ is Euclidean distance between objects $A$ and $B$.
Note that $D_r\subset C_r\subset R_r$. 
Let us choose some $r^0\in(0,1)$ such that $R_{r^0}\subset \Omega$.
For every $r\in (0,r^0)$ define a function $v_r$ by
\begin{eqnarray}
v_r(\bm{x})&=&0, \qquad \bm{x}\in\Omega\setminus R_r\nonumber\\
&=&1,\qquad \bm{x}\in D_r\nonumber\\
&=&1-\frac{d(\bm{x},D_r)}{r^2}\in \langle 0,1),
\quad \bm{x}\in R_r\setminus D_r.\nonumber
\end{eqnarray}
Note that $v_r$ is continuous in $\Omega$ and
\begin{equation}\nonumber
\nabla v_r(\bm{x}) =\left(\frac{\pm 1}{r^2},0,0\right),
\quad \bm{x}\in C_r\setminus D_r.
\end{equation}
The part $R_r\setminus C_r$ is a part of a torus.
Equipotential surfaces of $v_r$ defined in $R_r\setminus C_r$ 
are of the shape of toroidal surfaces again. Thus the gradient of $v_r$
in $\bm{x}\in R_r\setminus C_r$ directs to 
the nearest point $\bm{x}_D$ of $D_r$ to $\bm{x}$.
Due to the equidistant distribution of the equipotential surfaces
 the gradients have the same norms for any $\bm{x}$ 
of the same equipotential surface. Then
\begin{equation}\nonumber
 \left\vert \frac{\partial v_r}{\partial x_i}\right\vert\le \frac{1}{r^2},
\quad i=1,2,3,\quad\bm{x}\in R_r\setminus C_r.
\end{equation}
Volume of $R_r$ is less than $2\pi r^2(r+r^2)^2$, volume of 
$C_r$ is $2\pi r^4$, thus volume of $R_r\setminus C_r$ is less than
$2\pi r^5(2+r)$.
Then 
\begin{eqnarray}
\int_{C_r}\left( \frac{\partial v_r}{\partial x_1}\right)^2&=&2\pi  \nonumber\\
\int_{C_r}\left( \frac{\partial v_r}{\partial x_i}\right)^2&=&0,\quad i=2,3\nonumber\\
\int_{R_r\setminus C_r}\left( \frac{\partial v_r}{\partial x_i}\right)^2&\le& 2\pi r (2+r)
,\quad i=1,2,3,\nonumber
\end{eqnarray}
which yields
\begin{equation}\nonumber
\lim_{r\to 0}\Vert v_r\Vert_{\mathcal L}=\sqrt{2\pi}.
\end{equation}
The rest of the proof is analogous to the last part of the 
proof of~\cite[Lemma~2.1]{GNS2020}.
Denoting $u_r=(\lambda\mathcal{I}-\mathcal{L}^{-1}\mathcal{A})v_r$, we get
\begin{equation}\nonumber
\Vert u_r\Vert_{\mathcal L}^2=\langle {\mathcal L}u_r,u_r\rangle=
\langle (\lambda{\mathcal L}-{\mathcal A})v_r,u_r\rangle=
\int_\Omega \nabla u_r\cdot(\lambda I-K)\nabla v_r\le 
\left(\int_\Omega\Vert(\lambda I-K)\nabla v_r\Vert^2\right)^\frac{1}{2}
\Vert u_r\Vert_{\mathcal L}.
\end{equation}
Then
\begin{eqnarray}
\Vert u_r\Vert_{\mathcal L}^2 &\le&\int_\Omega (\lambda-\kappa_1)^2\left( \frac{\partial v_r}{\partial x_1}\right)^2+\sum_{i=2}^3
\int_\Omega (\lambda-\kappa_i)^2\left( \frac{\partial v_r}{\partial x_i}\right)^2\nonumber\\
&\le &(2\pi+2\pi r (2+r))\,\sup_{\bm{x}\in R_r}\vert\kappa_1(\bm{x}^0)-\kappa_1(\bm{x})\vert^2
+ 2\pi r (2+r)\sum_{i=2}^3 \sup_{\bm{x}\in R_r}\vert\kappa_1
(\bm{x}^0)-\kappa_i(\bm{x})\vert^2.
\nonumber
\end{eqnarray}
{From continuity of $K$ in $\bm{x}^0$} we get
\begin{equation}
\lim_{r\to 0} \Vert u_r\Vert_{\mathcal L}=0,\nonumber
\end{equation}
which concludes the proof of~\eqref{co}.
\end{proof}

\begin{lem}\label{lem23}
Let $\kappa_i$, $i=1,2,3$, be constant on an open subdomain 
$S\subset\Omega$. Assume 
\begin{equation}\nonumber
\sup_{\bm{x}\in\Omega}\kappa_j(\bm{x})<\inf_{\bm{x}\in\Omega}\kappa_m(\bm{x}),
\end{equation}
for some pair $j,m\in\{1,2,3\}$, $j\ne m$. Then 
\begin{equation}\nonumber
[\sup_{\bm{x}\in\Omega}\kappa_j(\bm{x}),\inf_{\bm{x}\in\Omega}\kappa_m(\bm{x})]\subset
{\rm sp } (\mathcal{L}^{-1}\mathcal{A}).
\end{equation}
\end{lem}
\begin{proof}
Let $j=1$ and $m=2$ without any loss of generality.
Denote $k_1=\kappa_1(\bm{x})$ and $k_2=\kappa_2(\bm{x})$,
$\bm{x}\in S$, the values of $\kappa_1$ and $\kappa_2$ in $S$.
Note that 
\begin{equation}\nonumber
[\sup_{\bm{x}\in\Omega}\kappa_1(\bm{x}),\inf_{\bm{x}\in\Omega}\kappa_2(\bm{x})]\subset
[k_1,k_2].
\end{equation}  
Let $\lambda\in (k_1,k_2)$ and choose some point $\bm{x}^0\in S$. 
Denote 
\begin{equation}\nonumber
S_h=(x_1^0,x_1^0+h\sqrt{\lambda-k_1})\times
(x_2^0,x_2^0+h\sqrt{k_2-\lambda})\times 
(x_3^0,x_3^0+h)
\end{equation}
where $h\in(0,1)$ is sufficiently small that $S_h\subset S$.
Then for any $n\in{\mathbb Z}$ the tensor product function 
\begin{equation}\nonumber
\phi(\bm{x})=\sin\left(\frac{n\pi(x_1-x_1^0)}{h\sqrt{\lambda - k_1}}\right)
\sin\left(\frac{n\pi(x_2-x_2^0)}{h\sqrt{k_2-\lambda}}\right)
\end{equation}
fulfills 
\begin{equation}\nonumber
(\lambda-k_1)\frac{\partial^2 \phi(\bm{x})}{\partial x_1^2}=
-\frac{n^2\pi^2}{h^2}\phi(\bm{x}),\quad
(k_2-\lambda)\frac{\partial^2 \phi(\bm{x})}{\partial x_2^2}=
-\frac{n^2\pi^2}{h^2}\phi(\bm{x}),\quad
(\lambda-k_3)\frac{\partial^2 \phi(\bm{x})}{\partial x_3^2}=0
\end{equation}
in $S_h$ and the boundary condition $\phi(\bm{x})=0$ on $\partial S_h$.
Thus the function $v\in H_0^1(\Omega)$ defined as 
\begin{eqnarray}
v(\bm{x})&=&\phi(\bm{x}), \quad \bm{x}\in S_h\nonumber\\
&=&0, \quad \bm{x}\notin S_h\nonumber
\end{eqnarray}
solves the weak form of the generalized eigenvalue problem~\eqref{p1}.
\end{proof}
\begin{lem}\label{lem24}
Assume that $K$ is continuous at least at a single point in $\Omega$.
Let 
\begin{equation}\nonumber
\sup_{\bm{x}\in\Omega}\kappa_j(\bm{x})<\inf_{\bm{x}\in\Omega}\kappa_m(\bm{x}),
\end{equation}
for some pair $j,m\in\{1,2,3\}$, $j\ne m$. Then 
\begin{equation}\nonumber
[\sup_{\bm{x}\in\Omega}\kappa_j(\bm{x}),\inf_{\bm{x}\in\Omega}\kappa_m(\bm{x})]\subset
{\rm sp } (\mathcal{L}^{-1}\mathcal{A}).
\end{equation}
\end{lem}
\begin{proof}
The proof is analogous to the proof of~\cite[Lemm~2.3]{GNS2020}.
\end{proof}

\begin{lem}\label{lem25}
Assume that $K$ is continuous in the closure $\overline\Omega$.
Then 
\begin{equation}\nonumber
{\rm sp}(\mathcal{L}^{-1}\mathcal{A})\subset {\rm Conv}(\cup_{i=1}^3\kappa_i(\overline{\Omega})).
\end{equation}
\end{lem}
\begin{proof}
The statement trivially follows from
\begin{equation}\nonumber
\min_{i=1,2,3}\inf_{\bm{x}\in\overline{\Omega}}\kappa_i(\bm{x})\int_\Omega\nabla\phi\cdot \nabla \phi\;\le \;
\int_\Omega\nabla\phi\cdot K\nabla \phi\;\le \;
\max_{i=1,2,3}\sup_{\bm{x}\in\overline{\Omega}}\kappa_i(\bm{x})
\int_\Omega\nabla\phi\cdot \nabla \phi
\end{equation}
for $\phi\in H_0^1(\Omega)$.
Full details can be found in the proof of~\cite[Lemm~2.4]{GNS2020}.
\end{proof}

\section{Discussion}

We introduce the proof of the main
results of~\cite{GNS2020} for 3D problems.
The main contribution is Lemma~\ref{lem22},
where a certain construction of a set of functions $v_r$ is presented.
The construction can be naturally generalized to higher dimensions. 
The methodology can also help
to derive estimates of eigenvalues of discretized operators,
and thus provide a link between preconditioned differential operators
and associated numerical linear algebra problems.
Especially, the functions $v_r$ can serve as approximations of eigenfunctions of the discretized preconditioned operators.

Some questions remain open: What is the distribution of 
eigenvaues of the associated preconditioned discretized 
operator?~or What can we say about the spectral estimates of 
preconditioned operators if
$K$ is piecewise constant?~and How these spectra depend 
on a discretization basis?

\vskip24pt

{\bf Acknowledgment.} The author thanks Tom\' a\v s Gergelits and
Zden\v ek Strako\v s for fruitful discussions and 
the Center of Advanced Applied Sciences, the European Regional Development Fund (project No.~CZ.02.1.01/0.0/0.0/16\_019/0000778) for a partial support.




\begin{thebibliography}{99}




\bibitem{G}
{T. Gergelits, K.-A. Mardal, B. F. Nielsen, Z. Strako\v s}, Laplacian preconditioning of elliptic PDEs: Localization of the eigenvalues of the discretized operator.
{\it SIAM Journal on Numerical Analysis}, {57}, 2019, pp.~1369--1394.


\bibitem{GNS2020}
{T. Gergelits, B. F. Nielsen, Z. Strako\v s}, 
Generalized spectrum of second order differential operators.
{\it SIAM Journal on Numerical Analysis}, {58}, 2020, pp.~2193--2211. 

\bibitem{LPZ2020}
M.~Ladeck\' y, I.~Pultarov\' a, J.~Zeman,
Guaranteed two-sided bounds on all eigenvalues of preconditioned
diffusion and elasticity problems solved by the finite element method.
{\it Applications of Mathematics}, 66(1), 2020, pp.~21--42.


\bibitem{NTH2009}
B.~F.~Nielsen, A.~Tveito, W.~Hackbusch, 
Preconditioning by inverting the Laplacian.
{\it IMA Journal of Numerical Analysis}, 29, 2009, pp.~24--42.

\bibitem{PL2021}
I.~Pultarov\' a, M.~Ladeck\' y,
Two-sided guaranteed bounds to individual eigenvalues
of preconditioned finite element and finite difference problems. 
{\it Numerical Linear Algebra with Applications}, 28(5), 2021, e2382.







\end{thebibliography}
\end{document}